\documentclass{rmmcart}
\usepackage{amssymb,amsmath,amsfonts,amsxtra}
\usepackage{verbatim}

\newcommand{\C}{\mathbb{C}}
\newcommand{\N}{\mathbb{N}}

\newcommand{\lpn}{\ell^{(p_n)}}

\newcommand{\topn}{^{p_n}}

\title[]{Isometric Operators on Variable-Exponent Discrete Lebesgue Spaces}
\author{Philip M. Gipson}
\address{Department of Mathematics \\ State University of New York College at Cortland \\ Cortland, NY 13045-0900}
\email{philip.gipson@cortland.edu}

\begin{document}
\begin{abstract}
We investigate the structure of norm-preserving and linear but not necessarily surjective operators on variable-exponent, discrete Lebesgue spaces. A certain class of isometries, novel to this work, are especially considered; this class completely coincides with all isometries when the Lebesgue space is classical, i.e.\ of a fixed-exponent. For said isometries it is shown that their actions are completely determined by pairs consisting of set-mappings and bounded functions on $\N$. This result recovers the previously-known structure of isometries on fixed-exponent spaces as a special case. In the second part, we show that another wide class of operators, including shift operators, are only isometric under very restrictive conditions on the exponent sequence. Together these results serve to highlight the striking similarities and yet radical differences between isometric operators on fixed- and variable-exponent spaces.
\end{abstract}
\maketitle

\section{Introduction}
The structure theory of isometric operators on the discrete Lebesgue spaces $\ell^p$ has been completely understood for some time. For $p=2$, the well-known Wold decomposition provides a thorough framework for understanding isometric operators on Hilbert space. For $p\not=2$, the work of Lamperti \cite{lamperti} provides a total description of isometric operators on general function spaces, including the Lebesgue spaces as specific examples. 

The variable-exponent Lebesgue spaces, denoted here by $\lpn$, form one natural generalization of the spaces $\ell^p$. The technical details for constructing these spaces will be left for Section \ref{prelims}, but the core idea is quite simple (and immediately conveyed by their name): allow the exponent used for norm-calculations to vary across orthogonal subspaces.

Because the $\ell^p$ spaces, and their norms in particular, have different behavior for $p>2$, $p<2$ and $p=2$, we might expect that variable-exponent spaces whose exponent sequence takes values in $[1,2)$ will exhibit different properties than those whose exponent sequence has values in $(2,\infty)$, or those whose exponents could take values in both ranges. Because of this, we will narrow our investigation to only those variable-exponent spaces $\lpn$ with $1\leq p_n<2$ for all $n\in\N$ or with $p_n>2$ for all $n\in\N$. In the latter case we will not necessarily assume that the exponent sequence is bounded.

Our primary result in the first part is Theorem \ref{mylampthm}. It gives a general structure theorem for ``isomodular" operators (a class of isometries, see Definition \ref{isomodulardef}) on $\lpn$ spaces whose exponent sequence satisfies either of the restrictions discussed above. This provides a strong similarity with the theory of isometric operators on classical Lebesgue spaces $\ell^p$, $p\not=2$.

\setcounter{theorem}{5}
\begin{theorem} If $(p_n)$ is such that either $p_n\in[1,2)$ for all $n\in\N$ or $p_n\in(2,\infty)$ for all $n\in\N$, and if $S:\lpn\to\lpn$ is an isomodular operator then there exists a regular set isomorphism $T$ and a function $h\in\ell^\infty$ such that
\[(Sx)_n=h(n)(Tx)_n\]
for all $x\in\lpn$ and $n\in\N$. Furthermore, $|h(n)|\leq1$ for all $n\in\N$.
\end{theorem}

In the second part we consider linear operators $S_\theta$ on variable-exponent spaces $\lpn$ which arise from measure-preserving actions $\theta$ on $\N$. This class of operators is of special interest because it contains the unilateral shift. Our primary result in this part is Theorem \ref{shiftisometry} and it gives a complete characterization of when such operators are isometric.

\setcounter{theorem}{6}
\begin{theorem} For given $\theta$ and $(p_n)$, the operator $S_\theta$ restricts to an isometry on $\lpn$ if and only if $p_n=p_{\theta(n)}$ for every $n\in\N$.
\end{theorem}

\setcounter{section}{1}
\setcounter{theorem}{0}

\section{Background}

In this section we will provide the background required to understand Lamperti's result for $\ell^p$ spaces, as it provides the inspiration for our investigation of the variable-exponent Lebesgue spaces $\lpn$. We will then provide a very specialized and narrow formulation of the theory of semimodular function spaces as they pertain to variable-exponent Lebesgue spaces.

\subsection{Isometries on $\ell^p$, $p\not=2$}
We will use the notation $\mathcal{P}(\N)$ to denote the powerset of $\N$. The following definition is given in full in \cite[\S 3]{lamperti} and is here restated only for the set $\N$ with the counting measure.
\begin{definition}\label{regsetiso} A \emph{regular set isomorphism} (of $\N$) is a mapping $T:\mathcal{P}(\N)\to\mathcal{P}(\N)$ satisfying:
\begin{enumerate}
\item $T(\N-A)=T\N-TA$
\item $T\left(\bigcup_{n=1}^\infty A_n\right)=\bigcup_{n=1}^\infty TA_n$
\item $TA=\emptyset\text{ if and only if }A=\emptyset$
\end{enumerate}
for all $A,A_n\in\mathcal{P}(\N)$. 
\end{definition}
Because of condition (ii), it is clear that there is a one-to-one correspondence between regular set isomorphisms $T$ on the one hand and denumerable collections of pairwise-disjoint subsets of $\N$ on the other. Explicitly, if $A_1,A_2,...\in\mathcal{P}(\N)$ is a denumerable and pairwise-disjoint collection then $T:\{n\}\mapsto A_n$ extends to a regular set isomorphism in the natural manner.

Given a regular set isomorphism $T$ and an element $a\in\ell^p$ we define $Ta\in\ell^\infty$ by 
$$(Ta)_n=\begin{cases} a_k &\text{ when }n\in T\{k\} \\ 0 &\text{ otherwise }\end{cases}.$$
Note that unless the cardinalities of $T\{k\}$, $k\in\N$, are finite and uniformly bounded then $Ta$ is not guaranteed to be in $\ell^p$ itself. 

\begin{theorem}\cite[Theorem 3.1, paraphrased]{lamperti}\label{oglamperti}
Let $U$ be a linear operator on $\ell^p$, where $p$ is a positive real number not equal to 2, such that
\[||Ua||=||a||\text{   for all }a\in\ell^p\]
(the norm is the usual $p$-norm). Then there exists a regular set isomorphism $T$ and an $h\in\ell^\infty$ such that
$$Ua=hTa.$$
\end{theorem}
We have omitted a second portion of the Theorem statement concerning certain measure-theoretic properties of $h$ which are not of interest in the $\ell^p$ formulation.

A key part of Lamperti's proof is the observation that an isometry $U$ will preserve orthogonality relations and, consequently, may be used to construct a collection of disjoint support sets which will define the regular set isomorphism. We will make use of a similar technique in our main result. 

\subsection{Variable Exponent Spaces}\label{prelims}
For our understanding of the variable exponent spaces $\lpn$ we will rely on the theory of \emph{modular spaces}. Our reference for this topic will be the excellent book \cite{diening} by Diening et al.

For any set $A$, the set of all possible $A$-valued functions on $\N$ will be denoted $A^\infty$. An element $p\in A^\infty$ will be referred to as a \emph{$A$-valued sequence} and we will use the common notations $p_n:=p(n)$ and $(p_n):=p$.

Let $p\in[1,\infty)^\infty$ (a $[1,\infty)$-valued sequence) be given and define
\[\lpn:=\left\{a\in\C^\infty\ :\ \sum_{n\in\N}|a_n|^{p_n}<\infty\right\}\]
It is nontrivial to determine that this is a complex vector space \cite[Chapter 2]{diening}. The sequences $e_k$, $k\in\N$, defined by $(e_k)_n=\delta_{kn}$ (here $\delta$ is the Kronecker delta) form a basis set for $\lpn$. 

The norm on $\lpn$ is given by
\[||a||:=\inf\left\{\lambda>0\ :\ \sum_{n\in\N}\left|\frac{a_n}{\lambda}\right|^{p_n} \leq 1\right\}.\]

In addition to the norm, we will define a \emph{modular} on $\lpn$, denoted $\varrho$, by
\[\varrho(a):=\sum_{n\in\N}|a_n|^{p_n}.\]

In the case when $p$ is a constant sequence we have that $\lpn$ is precisely the classical Lebesgue space $\ell^p$ and the modular and norm have the simple relationship $||a||^p=\varrho(a)$. 

In general, the modular and norm have only the relationship $$||a||=\inf\left\{\lambda>0\ :\ \varrho\left(\frac{a}\lambda\right)\leq 1\right\}.$$

In this more complicated situation it is not immediately clear whether norm-preserving, i.e.\ isometric, operators must necessarily preserve the modular. To that end, we offer the following definition:

\begin{definition}\label{isomodulardef} A linear operator $S:\lpn \to\lpn$ will be called \emph{isomodular} if $\varrho(Sa)=\varrho(a)$ for all $a\in\lpn$.
\end{definition}
Our first observation is that being isomodular is a stronger condition on an operator than being isometric.

\begin{proposition}\label{isomodimpliesisomet} If $S:\lpn \to\lpn$ is an isomodular operator then it is isometric.
\end{proposition}
\begin{proof} For $a\in\lpn$ and $\lambda >0$ we have $\varrho\left(\frac{a}\lambda\right)=\varrho\left(S(\frac{a}\lambda)\right)=\varrho\left(\frac{Sa}\lambda\right)$. Thus $\varrho\left(\frac{a}\lambda\right)\leq 1$ if and only if $\varrho\left(\frac{Sa}\lambda\right)\leq 1$. And hence
$$||Sa||=\inf\left\{\lambda>0:\varrho\left(\frac{Sa}\lambda\right)\leq 1\right\}=\inf\left\{\lambda>0:\varrho\left(\frac{a}\lambda \right)\leq 1\right\}=||a||$$
as desired.
\end{proof}
The converse statement is true in the special case of constant $(p_n)$, i.e.\ for classical $\ell^p$ spaces. This is a direct consequence of the simpler relation, $\varrho(a)=||a||^p$, between modular and norm which occurs in such spaces. It is currently unknown to us whether isometric operators are necessarily isomodular in general.

\section{Isomodulars on Spaces with Restricted Exponents}

Throughout this section we will assume that $(p_n)$ is such that either $p_n\in[1,2)$ for all $n\in\N$ or $p_n\in(2,\infty)$ for all $n\in\N$. The motivations for such a restriction on the exponent sequence are the classical Clarkson inequalities \cite{clarkson} for the usual (fixed-exponent) Lebesgue spaces $\ell^p$. The Clarkson inequalities quoted below are in an expanded form due to Lamperti \cite[Corollary 2.1]{lamperti}.

\begin{proposition} For $f,g\in \ell^p$ with $p>2$ we have $||f+g||^p+||f-g||^p\geq2||f||^p+2||g||^p$. If $p<2$ then the reverse inequality is true. In either case, equality occurs if and only if $fg=0$ almost everywhere.
\end{proposition}
Of course, if $p=2$ the Parallelogram Identity gives $||f+g||^2+||f-g||^2=2||f||^2+2||g||^2$ for any $f,g\in \ell^p$ regardless of orthogonality. Hence we explicitly exclude the possibility of our exponent sequences having 2 as a value.

Considering the above Proposition, for $z,w\in\C$ and $p>2$ we have $|z+w|^p+|z-w|^p\geq 2|z|^p+2|w|^p$ (and the reverse for $p<2$) with equality if and only if either $z=0$ or $w=0$. Therefore if $(p_n)$ is such that $p_n\in(2,\infty)$ for all $n\in\N$ and given $a,b\in\lpn$ then
$$|a_n+b_n|\topn+|a_n-b_n|\topn\geq2|a_n|\topn+2|b_n|\topn$$
with equality if and only if $a_nb_n=0$. It follows that
$$\sum_{n\in\N}|a_n+b_n|\topn+\sum_{n\in\N}|a_n-b_n|\topn\geq2\sum_{n\in\N}|a_n|\topn+2\sum_{n\in\N}|b_n|\topn$$
with equality if and only if $a_nb_n=0$ for all $n\in\N$. From the definition we now have that
$$\varrho(a+b)+\varrho(a-b)\geq2\varrho(a)+2\varrho(b)$$
 with equality if and only if $ab=0$. The reverse inequality holds if $(p_n)$ is such that $p_n\in[1,2)$ for all $n\in\N$, including the strict equality condition. Taken together these observations result in the following lemma.

\begin{lemma} If $(p_n)$ is such that either $p_n\in[1,2)$ for all $n\in\N$ or $p_n\in(2,\infty)$ for all $n\in\N$, and $a,b\in\lpn$ then
$$\varrho(a+b)+\varrho(a-b)=2\varrho(a)+2\varrho(b)$$
if and only if $ab=0$.
\end{lemma}

Recall that for the usual $\ell^p$ spaces the relation between the norm and modular is $||a||^p=\varrho(a)$ for all $a\in\ell^p$. Hence the Lemma is a direct generalization, to $\lpn$ spaces, of Lamperti's version of Clarkson's inequalities. Lamperti's primary use of Clarkson's inequalities was to conclude that isometric operators preserve orthogonal vectors. We find that isomodular operators, at least on spaces with restricted exponents, similarly preserve orthogonality.

\begin{proposition}\label{ortho} If $(p_n)$ is such that either $p_n\in[1,2)$ for all $n\in\N$ or $p_n\in(2,\infty)$ for all $n\in\N$, $S:\lpn \to\lpn$ is isomodular, and $a,b\in\lpn$ are such that $ab=0$ then $(Sa)(Sb)=0$ as well.
\end{proposition}
\begin{proof} Since $S$ is isomodular we have that $\varrho(Sx)=\varrho(x)$ for all $x\in\lpn$. It follows, using the previous Lemma, that
$$\varrho(Sa+Sb)+\varrho(Sa-Sb)=\varrho(a+b)+\varrho(a-b)=2\varrho(a)+2\varrho(b)=2\varrho(Sa)+2\varrho(Sb).$$
Using the previous Lemma again we immediately conclude that $(Sa)(Sb)=0$.
\end{proof}

The basis elements $e_n\in\lpn$ are, naturally, orthogonal, and so by the above proposition their images $Se_n$ under an isomodular operator are orthogonal as well. It follows that the support sets
$$\operatorname{support}(Se_n):=\{k\in\N\ :\ (Se_n)_k\not=0\}$$
must be pairwise disjoint. 

\begin{proposition} If $(p_n)$ is such that either $p_n\in[1,2)$ for all $n\in\N$ or $p_n\in(2,\infty)$ for all $n\in\N$ and if $S:\lpn\to\lpn$ is an isomodular operator then the assignments
$$T:\{n\}\mapsto\operatorname{support}(Se_n)$$
extend to a regular set isomorphism $T$ of $\N$. 
\end{proposition}

Recall from the discussion following Definition \ref{regsetiso} that if  $T$ is a regular set isomorphism of $\N$ then it extends to a map (also called $T$) from $\lpn$ to $\ell^\infty$.

We now have a sufficient preparation to state and prove our main result:

\begin{theorem}\label{mylampthm} If $(p_n)$ is such that either $p_n\in[1,2)$ for all $n\in\N$ or $p_n\in(2,\infty)$ for all $n\in\N$, and if $S:\lpn\to\lpn$ is an isomodular operator then there exists a regular set isomorphism $T$ of $\N$ and a function $h\in\ell^\infty$ such that
$$(Sx)_n=h_n(Tx)_n$$
for all $x\in\lpn$ and $n\in\N$. Furthermore, $|h(n)|\leq1$ for all $n\in\N$.
\end{theorem}
\begin{proof}
Define the regular set isomorphism $T$ as in the previous Proposition. Then the extension $T:\lpn \to\ell^\infty$ is given by
$$(Ta)_n=\begin{cases} a_k &\text{ when }n\in \operatorname{support}(Se_k)\\ 0 &\text{ otherwise }\end{cases}$$
where the values are well-defined precisely because the support sets are pairwise disjoint.

Consider $a,b\in\lpn$. Then each term of $T(a-b)$ is either zero or a difference of like terms, $(Ta-Tb)_n=a_k-b_k$, $n\in\operatorname{support}(Se_k)$. It follows that $||T(a-b)||_\infty\leq ||a-b||_\infty\leq||a-b||$. Hence $T:\lpn\to\ell^\infty$ is continuous.

In particular, $(Te_m)_n=1$ if $n\in\operatorname{support}Se_m$ and 0 otherwise, i.e.\ $Te_m$ is the indicator sequence for the support set of $Se_m$. Hence for any $m\in\N$ we have $Se_m=Se_mTe_m$.

Consider that $||e_m||=1$ for each $m\in\N$, and so by Proposition \ref{isomodimpliesisomet} $||Se_m||= 1$. Defining $h:=\sum Se_m$ we conclude that, because the $Se_m$ have pairwise disjoint support, $h\in\ell^\infty$. Further, for each $n\in\N$ we have that $|h(n)|=|(Se_m)_n|\leq ||Se_m||=1$ for some $m\in\N$.  

Consider that for given $k\in\N$ we have
$$Se_m=Se_mTe_m=\sum_{m\in\N}Se_mTe_m=hTe_m.$$
By linearity this extends to sequences $a\in\lpn$ with finite support. Such elements are dense in $\lpn$ \cite[Corollary 3.4.10]{diening} and so the conclusion follows by the continuity of $T$. 
\end{proof}

Our theorem precisely recovers Lamperti's result in the case when $(p_n)$ is constant, and evokes a result \cite{felmingjamison}[Theorem 9.2.12] for \emph{surjective} isometries on a wider class or sequence spaces.

\section{Operators Induced by Transformations of $\N$}

We have already seen that an isomodular operator is necessarily isometric. Since isomodular operators on variable-exponent spaces with restricted exponents are now known, per Theorem \ref{mylampthm}, to have a determined structure, so too do a certain class of isometries. The goal for our current section is to understand the structure of another potential class of isometries: shift operators.

The unilateral shift $S:\C^\infty\to\C^\infty$ is defined for each $a\in\C^\infty$ by $(Sa)_1=0$ and $(Sa)_{n+1}=a_n$ for $n\in\N$. For a general variable-exponent space $\lpn\subset\C^\infty$ there is no guarantee that $S$ is a bounded operator, or even that $Sa\in\lpn$ when $a\in\lpn$. 

\begin{example}\label{example1} Take $(p_n)=(1,2,1,2,...)$ and $a=(0,1,0,\frac12,0,\frac13,...)$. Then $a\in\lpn$ but $Sa$ is not.
\end{example}
If $\Gamma:\N\to\N$ is a bijection then define a linear operator $S_\Gamma:\C^\infty\to\C^\infty$ where for each $a\in\C^\infty$ we have $(S_\Gamma a)_n=a_{\Gamma^{-1}(n)}$. When acting on a classical $\ell^p\subset\C^\infty$ these $S_\Gamma$ operators are isometric isomorphisms. However, there is no such guarantee when they act on a variable-exponent space.

\begin{example} Take $(p_n)$ and $a\in\lpn$ as in Example \ref{example1}. Set $\Gamma(n)=n-(-1)^n$. Then $S_\Gamma a\not\in\lpn$. 
\end{example}

Generalizing both previous concepts, consider operators $S_\theta:\C^\infty\to\C^\infty$, related to injective maps $\theta:\N\to\N$, which are defined by
$$(S_\theta a)_n:=\begin{cases}a_{\theta^{-1}(n)} & n\in\theta(\N) \\ 0 & \text{otherwise}\end{cases}$$
Notice that $S_\theta e_n=e_{\theta(n)}$ for each $n\in\N$; this could also be considered the defining feature of $S_\theta$. 

\begin{theorem}\label{shiftisometry} For given $\theta$ and $(p_n)$, the operator $S_\theta$ restricts to an isometry on $\lpn$ if and only if $p_n=p_{\theta(n)}$ for every $n\in\N$.
\end{theorem}
\begin{proof}
For necessity, let $a\in\lpn$ be given and suppose that $p_{\theta(n)}=p_n$ for each $n\in\N$. Consider that
$$||S_\theta a||=\inf_{\theta>0}\left\{\sum_{n\in\N}\left|\frac{(S_\theta a)_n}{\lambda}\right|^{p_n}\leq 1\right\}$$
For sufficiency the claim is trivial if $\theta$ is the identity mapping, so let $j\in\N$ be such that $\theta(j)\not=j$.

Consider the sequence $b:=2^{-\frac1{p_j}}e_j+2^{-\frac1{p_{\theta(j)}}}e_{\theta(j)}$. By definition we have that 
$$||b||=\inf\left\{\lambda>0\ :\ \frac{\lambda^{-p_j}}{2}+\frac{\lambda^{-p_{\theta(j)}}}{2}\leq 1\right\}.$$
As before, it must be that $||b||$ satisfies $\frac{||b||^{-p_j}}{2}+\frac{||b||^{-p_{\theta(j)}}}{2}= 1$, however since $p_j,p_{\theta(j)}\in[1,\infty)$ we conclude this is only possible when $||b||=1$. 

If $||S_\theta b||=||b||=1$ then $\lambda =1$ is a solution to
$$\left(\frac{2^{-\frac1{p_{j}}}}{\lambda}\right)^{p_{\theta(j)}}+\left(\frac{2^{-\frac1{p_{\theta(j)}}}}{\lambda}\right)^{p_{\theta(\theta(j))}}=1$$
thus
$$2^{-\frac{p_{\theta(j)}}{p_j}}+2^{-\frac{p_{\theta(\theta(j))}}{p_{\theta(j)}}}=1.$$
Since the equation $2^x+2^y=1$ has the unique (real) solution $x=y=-1$ we must conclude that $p_j=p_{\theta(j)}=p_{\theta(\theta(j))}$.
\end{proof}

Our Theorem partially recalls with the result in \cite{skorik} where it is shown that the bijective isometries of real-valued analogs to $\lpn$ are precisely obtained from entry-value-preserving permutations of $(p_n)$ together with sign-changes. 

The injection $\theta(n):=n+1$ yields $S_\theta=S$ the unilateral shift; and so we have the following:

\begin{corollary} The unilateral shift is isometric on $\lpn$ if and only if $(p_n)$ is constant.
\end{corollary}

Similarly we find any multiple of the shift, $S^k$, is isometric on $\lpn$ if and only if $(p_n)$ is periodic and the period divides $k$.

\end{document}